\documentclass[preprint, 1p]{elsarticle}
\usepackage[utf8]{inputenc}

\usepackage{amsmath}
\usepackage{amsthm}
\usepackage{ulem}
\usepackage{enumitem}

\makeatletter
\newcommand{\myitem}[1][]{
  \protected@edef\@currentlabel{#1}%
\item[#1]
}
\makeatother

\DeclareMathOperator{\modf}{mod}

\newcommand{\op}{{^\textrm{op}}}

\newcommand{\HomT}[1]{\mathscr{T}(T,#1)}
\newcommand{\Db}{\mathscr{D}^b}

\newcommand{\D}{\mathscr{D}}

\newcommand{\T}{\mathscr{T}}

\usepackage{tikz}
\usetikzlibrary{arrows,decorations.pathmorphing,decorations.pathreplacing, decorations.markings, matrix, patterns} 

\usepackage{mathdots}

\usepackage{color}

\usepackage{pdfsync}

\usepackage{enumitem}

\usepackage{wasysym}

\usepackage{amssymb}

\usepackage{mathrsfs}

\DeclareMathAlphabet{\mathpzc}{OT1}{pzc}{m}{it}

\usepackage[all]{xy}

\usepackage{tikz}
\usetikzlibrary{arrows,matrix,decorations.pathmorphing,decorations.pathreplacing,positioning,shapes.geometric,shapes.misc,decorations.markings,decorations.fractals,calc,patterns}

\usepackage{graphicx}
\usepackage{float}
\usepackage[bottom]{footmisc}
\usepackage{moreenum}
\usepackage{makecell}
\entrymodifiers={+!!<0pt,\fontdimen22\textfont2>}

\def\cD{\mathscr{D}}

\def\cF{\mathscr{F}}

\def\cO{\mathscr{O}}

\def\cT{\mathscr{T}}

\def\cX{\mathscr{X}}

\def\BZ{\mathbb{Z}}

\def\add{\operatorname{add}}

\def\adots{\mathinner{\mkern1mu\raise1.0pt\vbox{\kern7.0pt\hbox{.}}\mkern2mu\raise4.0pt\hbox{.}\mkern2mu\raise7.0pt\hbox{.}\mkern1mu}}

\def\D{\cD}

\def\dddots{\mathinner{\mkern1mu\raise10.0pt\vbox{\kern7.0pt\hbox{.}}\mkern2mu\raise5.3pt\hbox{.}\mkern2mu\raise1.0pt\hbox{.}\mkern1mu}}
\def\dddotssmall{\mathinner{\mkern1mu\raise7.0pt\vbox{\kern7.0pt\hbox{.}}\mkern-1mu\raise4pt\hbox{.}\mkern-1mu\raise1.0pt\hbox{.}\mkern1mu}}

\def\Db{\cD^{\operatorname{b}}}

\def\End{\operatorname{End}}

\def\Ext{\operatorname{Ext}}

\def\H{\operatorname{H}}

\def\Hom{\operatorname{Hom}}

\def\inj{\operatorname{inj}}

\def\Ker{\operatorname{Ker}}

\def\mod{\operatorname{mod}}

\def\proj{\operatorname{proj}}

\def\PSL2{\operatorname{PSL}_2}

\def\rad{\operatorname{rad}}

\def\SL2{\operatorname{SL}_2}

\numberwithin{equation}{section}

\newtheorem{Lemma}{Lemma}[section]
\newtheorem{Theorem}[Lemma]{Theorem}
\newtheorem{Proposition}[Lemma]{Proposition}
\newtheorem{Corollary}[Lemma]{Corollary}

\theoremstyle{definition}
\newtheorem{Definition}[Lemma]{Definition}
\newtheorem{Setup}[Lemma]{Setup}

\newtheorem{Example}[Lemma]{Example}

\newtheorem*{bfhpg*}{}

\newenvironment{VarDescription}[1]%
  {\begin{list}{}{%
    \settowidth{\labelwidth}{\textbf{#1:}}%
    \setlength{\leftmargin}{\labelwidth}\addtolength{\leftmargin}{\labelsep}}}%
  {\end{list}}

\begin{document}

%\normalem
%\setlength{\parindent}{0pt}
%\setlength{\parskip}{7pt}

\title{$d$-abelian quotients of $(d+2)$-angulated categories}

\author[kmj]{Karin M.\ Jacobsen\corref{cor1}}
\ead{karin.jacobsen@ntnu.no}
\ead[url]{https://www.ntnu.edu/employees/karin.jacobsen}

\address[kmj]{Norwegian University of Science and Technology, 
Department of Mathematical Sciences, 
Sentralbygg 2, Gl\o shaugen, 
7491 Trondheim, Norway}

\author[pj]{Peter J\o rgensen}
\ead{peter.jorgensen@ncl.ac.uk}
\ead[url]{http://www.staff.ncl.ac.uk/peter.jorgensen}
\address[pj]{School of Mathematics and Statistics,
Newcastle University, Newcastle upon Tyne NE1 7RU, United Kingdom}

\cortext[cor1]{Corresponding author}

%\thanks{Date: \today. A thank you would go here}

\begin{keyword}
Cluster tilting object\sep $d$-abelian category\sep $d$-cluster tilting subcategory\sep $d$-representation finite algebra\sep $( d+2 )$-angulated category\sep functorially finite subcategory\sep Gorenstein algebra\sep higher homological algebra\sep quotient category\sep quotient functor
\end{keyword}

%\subjclass[2010]{16G10, 18E10, 18E30, 18E35}

%05B45: Tessellation and tiling problems 
%05E10: Combinatorial aspects of representation theory
%05E15: Combinatorial aspects of groups and algebras
%05E40: Combinatorial aspects of commutative algebra
%05E45: Combinatorial aspects of simplicial complexes
%05E99: Algebraic combinatorics / None of the above, but in this section
%13D25: Complexes
%13F60: Cluster algebras
%16E10: Homological dimension
%16E45: Differential graded algebras and applications
%16G10: Representations of Artinian rings 
%16G60: Representation type (finite, tame, wild, etc.) 
%16G70: Auslander-Reiten sequences (almost split sequences) and
%       Auslander-Reiten quivers
%16S90: Torsion theories; radicals on module categories
%18A20: Epimorphisms, monomorphisms, special classes of morphisms, null morphisms
%18E10: Exact categories, abelian categories
%18E30: Derived categories, triangulated categories
%18E35: Localization of categories
%18E40: Torsion theories, radicals
%18G05: Projectives and injectives
%18G35: Chain complexes
%18G99: Homological algebra: None of the above, but in this section 
%51M20: Polyhedra and polytopes; regular figures, division of spaces
%55P62: Rational homotopy theory

\begin{abstract}

Let ${\mathscr T}$ be a triangulated category.  If $T$ is a cluster tilting object and $I = [ \operatorname{add} T ]$ is the ideal of morphisms factoring through an object of $\operatorname{add} T$, then the quotient category ${\mathscr T} / I$ is abelian.  This is an important result of cluster theory, due to Keller--Reiten and K\"{o}nig--Zhu.  More general conditions which imply that ${\mathscr T} / I$ is abelian were determined by Grimeland and the first author.

Now let ${\mathscr T}$ be a suitable $( d+2 )$-angulated category for an integer $d \geqslant 1$.  If $T$ is a cluster tilting object in the sense of Oppermann--Thomas and $I = [ \operatorname{add} T ]$ is the ideal of morphisms factoring through an object of $\operatorname{add} T$, then we show that ${\mathscr T} / I$ is $d$-abelian.

The notions of $( d+2 )$-angulated and $d$-abelian categories are due to Geiss--Keller--Oppermann and Jasso.  They are higher homological generalisations of triangulated and abelian categories, which are recovered in the special case $d = 1$.  We actually show that if $\Gamma = \operatorname{End}_{ \mathscr T }T$ is the endomorphism algebra of $T$, then ${\mathscr T} / I$ is equivalent to a $d$-cluster tilting subcategory of $\operatorname{mod} \Gamma$ in the sense of Iyama; this implies that ${\mathscr T} / I$ is $d$-abelian.  Moreover, we show that $\Gamma$ is a $d$-Gorenstein algebra.

More general conditions which imply that ${\mathscr T} / I$ is $d$-abelian will also be determined, generalising the triangulated results of Grimeland and the first author.

\end{abstract} 

\maketitle

\setcounter{section}{-1}
\section{Introduction}
\label{sec:introduction}

It is an important result of cluster theory that certain quotients of triangulated categories are abelian.  This is stated in theorems by
Keller--Reiten, K\"{o}nig--Zhu, and in \cite[thm.\ 1]{GJ}, which will be generalised here to $( d+2 )$-angulated and $d$-abelian categories, the basic objects of higher homological algebra.

\subsection{Classic background}

Let $\cT$ be a $k$-linear $\Hom$-finite triangulated category over a field $k$, and let $T \in \cT$ be an object with endomorphism algebra $\Gamma = \End_{ \cT }T$.  Denote by $\cD$ the essential image of the functor $\cT( T,- ) : \cT \rightarrow \modf \Gamma$.

Recall the notion of cluster tilting objects (also known as maximal $1$-orthogonal objects), which was introduced by Iyama, see \cite[def.\ 3.1]{KZ}.  In our setup, $T$ is cluster tilting if it satisfies:
\[
  \add T
  = \{ X \in \cT \mid \cT( T,\Sigma X ) = 0 \}
  = \{ X \in \cT \mid \cT( X,\Sigma T ) = 0 \},
\]
where $\Sigma$ is the suspension functor of $\cT$.  When $T$ is cluster tilting, each $X \in \cT$ permits what might be called a $T$-presentation: A triangle $T_1 \rightarrow T_0 \rightarrow X \rightarrow \Sigma T_1$ with the $T_i$ in $\add T$.  See \cite[sec.\ 2.1, proposition]{KellerR} and \cite[lem.\ 3.2.1]{KZ}.  

It follows that a cluster tilting object $T$ satisfies the following conditions:
\begin{enumerate}
\setlength\itemsep{4pt}

  \myitem[(a)]  \label{condsGJ}
Let $T_1 \xrightarrow{f} T_0$ be a right minimal morphism in $\add T$.  Then each completion of $f$ to a triangle $T_1 \xrightarrow{f} T_0 \rightarrow X \xrightarrow{h} \Sigma T_1$ in $\cT$ satisfies $\cT( T,h ) = 0$.
  
  \myitem[(b)]  Let $X \in \cT$ be indecomposable with $\cT( T,X ) \neq 0$.  Then there exists a triangle $T_1 \rightarrow T_0 \rightarrow X \xrightarrow{h} \Sigma T_1$ in $\cT$ which satisfies $\cT( T,h ) = 0$.
\end{enumerate}
Note that (a) and (b) do not imply that $T$ is cluster tilting, see \cite[exa.\ 18]{GJ}.  We are interested in (a) and (b) because of the following result:

\begin{Theorem}
[{\cite[thm.\ 1]{GJ}}]
\label{thm:GJ}
Conditions (i) and (ii) below are equivalent.
\begin{enumerate}
\setlength\itemsep{4pt}

  \item  The functor $\cT( T,- ) : \cT \rightarrow \modf \Gamma$ is 
essentially surjective (in other words: $\cD = \modf \Gamma$), and it is full.

  \item  $T$ satisfies conditions (a) and (b).

\end{enumerate}
\end{Theorem}

If (i) holds, then $\cT( T,- ) : \cT \rightarrow \modf \Gamma$ induces an equivalence of categories
\[
  \cT/I \xrightarrow{ \sim } \modf \Gamma
\]  
where $I$ is the ideal of morphisms $f$ such that $\cT( T,f ) = 0$.  In other words, the triangulated category $\cT$ has an abelian quotient $\cT / I$.

If $T$ is cluster tilting, then more is true.  The following is a combination of \cite[sec.\ 2.1]{KellerR} and \cite[cors.\ 4.4 and 4.5]{KZ}:

\begin{Theorem}
[Keller--Reiten and K\"{o}nig--Zhu]
\label{thm:classic}
Assume that $T$ is cluster tilting.  Then:
\begin{enumerate}
\setlength\itemsep{4pt}

  \item  The functor $\cT( T,- ) : \cT \rightarrow \modf \Gamma$ is 
essentially surjective (in other words: $\cD = \modf \Gamma$).
  
  \item  The functor $\cT( T,- )$ induces an equivalence of categories
\[  
  \cT / [ \add \Sigma T ] \xrightarrow{ \sim } \modf \Gamma.
\]

  \item $\Gamma$ is a $1$-Gorenstein algebra, that is, each injective module has projective dimension $\leqslant 1$, and each projective module has injective dimension $\leqslant 1$.

  \item  If the global dimension of $\Gamma$ is finite, then it is at most $1$.

\end{enumerate}
\end{Theorem}

The purpose of this paper is to generalise Theorems \ref{thm:GJ} and \ref{thm:classic} to $( d+2 )$-angulated categories.

\subsection{Primer on $( d+2 )$-angulated and $d$-abelian categories}

The notions of $( d+2 )$-angulated and $d$-abelian categories were introduced by Geiss--Keller--Oppermann in \cite[def.\ 2.1]{GKO} and Jasso in \cite[def.\ 3.1]{Jasso}.  They are the basic objects of higher homological algebra.
For $d = 1$ they specialise to triangulated and abelian categories.  For general values of $d$, they are defined in terms of $( d+2 )$-angles, $d$-kernels, and $d$-cokernels; these are longer complexes with properties resembling those of triangles, kernels, and cokernels.

Many examples of $( d+2 )$-angulated and $d$-abelian categories are known, see  for instance \cite{GKO}, \cite{Jasso}, \cite{OppermannT}, and Section \ref{sec:examples}, and there are strong links to higher dimensional combinatorics.

The notion of cluster tilting object can be generalised to $( d+2 )$-angulated categories:

\begin{Definition}
[{Cluster tilting objects in the sense of  \cite[def.\ 5.3]{OppermannT}}]
\label{def:OT}
An object $T$ of a $( d+2 )$-angulated category $\cT$ with $d$-suspension functor $\Sigma^d$ is called {\em cluster tilting in the sense of Oppermann--Thomas} if:
\begin{enumerate}
\setlength\itemsep{4pt}

  \item  \(\HomT{\Sigma^d T}=0.\)
  
  \item  Each \(X\in \cT\) occurs in a $( d+2 )$-angle
\[
  T_d \rightarrow T_{d-1} \rightarrow \cdots \rightarrow T_{1} \rightarrow T_0 \xrightarrow{f_0} X \xrightarrow{h} \Sigma^dT_d
\]
with \(T_i\in \add T\) for $0 \leqslant i \leqslant d$.

\end{enumerate}
\end{Definition}

\subsection{This paper}

This paper generalises Theorems \ref{thm:GJ} and \ref{thm:classic} to $( d+2 )$-angulated categories.  We first fix the notation.  Concrete examples of the following setup are provided in Section \ref{sec:examples}.

\begin{Setup}
\label{set:blanket0}
The rest of the paper assumes the following setup:  $k$ is an algebraically closed field, $d \geqslant 1$ is an integer, $\cT$ is a $k$-linear $\Hom$-finite $( d+2 )$-angulated category with split idempotents.  The $d$-suspension functor of $\cT$ is denoted by $\Sigma^d$.  We assume that $\cT$ has a Serre functor $S$, that is, an autoequivalence for which there are natural equivalences $D\cT( X,Y ) \cong \cT( Y,SX )$, where $D( - ) = \Hom_k( -,k )$ is the $k$-linear duality functor.

We let $T \in \cT$ be an object with endomorphism algebra $\Gamma=\End_{ \cT } T$.  By $\cD$ we denote the essential image of the functor $\cT( T,- ) : \cT \rightarrow \modf \Gamma$, where $\modf \Gamma$ is the category of finite dimensional right $\Gamma$-modules.

Observe that since $\cT$ is $k$-linear and $\Hom$-finite, it is a Krull--Schmidt category. 
\hfill $\Box$
\end{Setup}

Our first main result is a higher homological generalisation of Theorem \ref{thm:GJ}, which can be recovered by setting $d = 1$.  Conditions \ref{condGhostWeak}, \ref{condGhostDualWeak}, \ref{condGhost}, \ref{condGhostDual}, and \ref{condRes} in the theorem are higher homological versions of conditions (a) and (b) on page \pageref{condsGJ}.  We do not state them here, but refer to Definition \ref{def:conditions}.

\begin{Theorem}
\label{thm:C}
Conditions (i), (ii), and (iii) below are equivalent.
\begin{enumerate}
\setlength\itemsep{4pt}

  \item  $\cD$ is a $d$-cluster tilting subcategory of $\modf \Gamma$ (see Definition \ref{def:d-CT} below) and the functor $\cT( T,- ) : \cT \rightarrow \modf \Gamma$ is full.

  \item  $T$ satisfies conditions \ref{condGhostWeak}, \ref{condGhostDualWeak}, and \ref{condRes} in Definition \ref{def:conditions}.

  \item  $T$ satisfies conditions \ref{condGhost}, \ref{condGhostDual}, and \ref{condRes} in Definition \ref{def:conditions}.

\end{enumerate}
\end{Theorem}

If (i) holds, then $\cD$ is a $d$-cluster tilting subcategory of $\modf \Gamma$, hence $d$-abelian by \cite[thm.\ 3.16]{Jasso}.  Moreover, $\cT( T,- ) : \cT \rightarrow \modf \Gamma$ induces an equivalence of categories
\[
  \cT / I \xrightarrow{ \sim } \cD,
\]
where $I$ is the ideal of morphisms $f$ such that $\cT( T,f ) = 0$.  In other words, the $( d+2 )$-angulated category $\cT$ has a $d$-abelian quotient $\cT / I$.

Let us remark that the implication (iii)$\Rightarrow$(ii) in the theorem is clear by Definition \ref{def:conditions}, since conditions \ref{condGhost}, \ref{condGhostDual} are explicitly stronger versions of conditions \ref{condGhostWeak}, \ref{condGhostDualWeak}.  The implications (ii)$\Rightarrow$(i) and (i)$\Rightarrow$(iii) will be proved in Sections \ref{sec:ii_to_i} and \ref{sec:i_to_iii}, respectively.

Our second main result is a higher homological generalisation of Theorem \ref{thm:classic}, which can be recovered by setting $d = 1$.  Note that the following was obtained in a special case in the first part of \cite[thm.\ 5.6]{OppermannT}.

\begin{Theorem}
\label{thm:AB}
Assume that $T$ is cluster tilting in the sense of Oppermann--Thomas, see Definition \ref{def:OT}.  Then:
\begin{enumerate}
\setlength\itemsep{4pt}

  \item  $\cD$ is a $d$-cluster tilting subcategory of $\modf \Gamma$.
  
  \item  The functor $\cT( T,- )$ induces an equivalence of categories
\[  
  \cT / [ \add \Sigma^d T ] \xrightarrow{ \sim } \cD.
\]

  \item $\Gamma$ is a $d$-Gorenstein algebra, that is, each injective module has projective dimension $\leqslant d$, and each projective module has injective dimension $\leqslant d$.

  \item  If the global dimension of $\Gamma$ is finite, then it is at most $d$.

\end{enumerate}
\end{Theorem}

From Theorem \ref{thm:AB} follows the next result, which was obtained in a special case in the second part of \cite[thm.\ 5.6]{OppermannT}.  The notion of (weakly) $d$-representation finite algebras was defined in  \cite[def.\ 2]{IO}.

\begin{Corollary}
\label{cor:repfin}
Assume that $T$ is cluster tilting in the sense of Oppermann--Thomas, see Definition \ref{def:OT}, and that $\cT$ has finitely many indecomposable objects up to isomorphism.  Then:
\begin{enumerate}
\setlength\itemsep{4pt}

  \item  $\Gamma$ is weakly $d$-representation finite.
  
  \item  If $\Gamma$ has finite global dimension, then it is $d$-representation finite.

\end{enumerate}
\end{Corollary}

The paper is organised as follows:  Section \ref{sec:lemmas1} provides some lemmas on $d$-cluster tilting subcategories of $\modf \Gamma$.  Section \ref{sec:lemmas2} provides some lemmas on the functor $\cT( T,- )$.  Section \ref{sec:conditions} states conditions \ref{condGhostWeak}, \ref{condGhostDualWeak}, \ref{condRes}, \ref{condGhost}, and \ref{condGhostDual}, and provides a connection to cluster tilting in the sense of Oppermann--Thomas.  Sections \ref{sec:ii_to_i} and \ref{sec:i_to_iii} prove the implications (ii)$\Rightarrow$(i) and (i)$\Rightarrow$(iii) in Theorem \ref{thm:C}.  Section \ref{sec:proofs} proves Theorem \ref{thm:AB} and Corollary \ref{cor:repfin}.  Section \ref{sec:examples} provides two classes of examples, the first of which shows how Theorem \ref{thm:AB} and Corollary \ref{cor:repfin} imply \cite[thm.\ 5.6]{OppermannT}.

\section{Lemmas on $d$-cluster tilting subcategories of $\mod \Gamma $}
%\section{Lemmas on $d$-cluster tilting subcategories of $\modf{\Gamma}$ test}
\label{sec:lemmas1}

The results of this section do not require $\Gamma$ to arise as in Setup \ref{set:blanket0}; they are valid for any finite dimensional $k$-algebra.

\begin{Definition}
[{$d$-cluster tilting subcategories, \cite[def.\ 1.1]{Iyama}}]
\label{def:d-CT}
Let $\cX \subseteq \modf \Gamma$ be a full subcategory.
\begin{enumerate}
\setlength\itemsep{4pt}

  \item  $\cX$ is {\em weakly $d$-cluster tilting} if
\begin{align*}
  \cX = & \: \{ X \in \modf \Gamma \mid \mbox{$\Ext_{ \Gamma }^i( X,\cX ) = 0$ for $1 \leqslant i \leqslant d-1$} \} \\
      = & \: \{ X \in \modf \Gamma \mid \mbox{$\Ext_{ \Gamma }^i( \cX,X ) = 0$ for $1 \leqslant i \leqslant d-1$} \}.
\end{align*}

  \item  $\cX$ is {\it $d$-cluster tilting} if it is weakly $d$-cluster tilting and functorially finite in $\modf \Gamma$.  

\end{enumerate}

A module $X \in \modf \Gamma$ is called {\em $d$-cluster tilting} if $\add X$ is a $d$-cluster tilting subcategory.
\end{Definition}

\begin{Setup}
From now on, $\cX \subseteq \modf \Gamma$ is a $d$-cluster tilting subcategory.  Note that $\cX$ is a $d$-abelian category by \cite[thm.\ 3.16]{Jasso}.
\end{Setup}

\begin{Lemma}
\label{non-proj-syzygy}
For \(1\leqslant i \leqslant d-1\) and \(X\in \cX\), the \(i\)th syzygy \(\omega^iX\), as defined by a minimal projective resolution of \(X\), has no non-zero projective summands. 
\end{Lemma}

\begin{proof}
Assume to the contrary that \(\omega^i X=\widetilde{\omega^i X}\oplus Q\) with \(Q\) non-zero projective.  Consider the augmented minimal projective resolution with syzygies:
\begin{center}
\begin{tikzpicture}[scale=1.3]
\node (enddots) at (0,0) {\(\cdots\)};
\node (Pi) at (1,0) {\(P_i\)};
\node (Pi-1) at (3,0) {\(P_{i-1}\)};
%\node (Pi-2) at (4.5,0) {\(P_{i-2}\)};

\node (middots) at (5, 0) {\(\cdots\)};
\node (P0) at (6,0) {\(P_0\)};
\node (X) at (7,0) {\(X\)};
\node (0) at (8,0) {\(0.\)};

\node (syzi) at (2, -1) {\(\widetilde{\omega^iX}\oplus Q\)};
\node (syzi-) at (4, -1) {\(\omega^{i-1}X\)};

\draw[->] (enddots) -- (Pi);
\draw[->] (Pi) -- (Pi-1);
%\draw[->] (Pi-1) -- (Pi-2);
\draw[->] (Pi-1) -- (middots);
\draw[->] (middots) -- (P0);
\draw[->] (P0) -- (X);
\draw[->] (X) -- (0);

\draw[->] (Pi) -- (syzi);
\draw[->] (syzi) -- node[below, sloped, font=\footnotesize]{\tiny \((u,v)\)} (Pi-1);
\draw[->] (Pi-1) --node[below, sloped, font=\footnotesize]{\tiny \(p_{i-1}\)} (syzi-);
\draw[->] (syzi-) -- (middots);
\end{tikzpicture}
\end{center}
Since \(X,Q\in \cX\), we have \(\Ext^i_\Gamma(X,Q)=0\).  Hence the map \((0, 1_Q): \widetilde{\omega^i X}\oplus Q \rightarrow Q\) must factor through \((u,v)\), so there is a map \(w: P_{i-1}\rightarrow Q\) with \((0, 1_Q) = w\circ(u,v)\). In particular \(1_Q = w\circ v\), whence \(1_Q-wv=0\), so \(v\notin\rad_{\modf \Gamma}\). This contradicts that \(p_{i-1}\) is a projective cover.
\end{proof}

%This is a dilemma, in that it's a lemma that proves two statements.
\begin{Lemma}
\label{left-min-proj-res}
Let \(X\in \cX\) have the augmented minimal projective resolution
\[
  \cdots \rightarrow P_2 \xrightarrow{f_2}P_1\xrightarrow{f_1}P_0\rightarrow X\rightarrow 0.
\]
\begin{enumerate}%[label=\emph{\textbf{\roman*}}]
\setlength\itemsep{4pt}

\item If \(2\leqslant j\leqslant d\) then \(f_j\) is left minimal.

\item If \(X\) has no non-zero projective summands, then \(f_1\) is left minimal.

\end{enumerate}
\end{Lemma}

\begin{proof}
(i):  Suppose \(g:P_{j-1}\rightarrow P_{j-1}\) satisfies \(gf_j=f_j\). 
Let \(p_{j-1}:P_{j-1}\rightarrow \omega^{j-1}X\) be the projective cover of the \((j-1)\)th syzygy. 
Since \((g-1_{P_{j-1}})f_j=0\), there must exist \(h: \omega^{j-1}X\rightarrow P_{j-1}\) such that \(g-1_{P_{j-1}}=hp_{j-1}\). 
In other words, \(g=1_{P_{j-1}}+hp_{j-1}\). But Lemma \ref{non-proj-syzygy} implies that \(p_{j-1}\) is in the radical, so \(g\) is invertible.

(ii): Use the same argument as for (i) with \(f_1\) in place of \(f_j\).
\end{proof}

\begin{Lemma}\label{projres-d-cokernel}
If \(X\in \cX\) has the augmented projective resolution
\[
  \cdots \xrightarrow{f_2} P_1 \xrightarrow{f_1} P_0\xrightarrow{f_0} X\rightarrow 0,
\]
then
\[
  P_{r-1} \xrightarrow{f_{r-1}} P_{r-2} \xrightarrow{f_{r-2}} \cdots \xrightarrow{f_2} P_1 \xrightarrow{f_1} P_0 \xrightarrow{f_0} X \rightarrow 0
\rightarrow \cdots \rightarrow 0
\]
is a \(d\)-cokernel of \(f_{r}\) in $\cX$ for each $1 \leqslant r \leqslant d$.  (For the definition of $d$-cokernels see \cite[def.\ 2.2]{Jasso}.)
\end{Lemma}

\begin{proof}
By the definition of $d$-cokernels, we must show that the complex
\[
  P_r \xrightarrow{f_r} P_{r-1} \xrightarrow{f_{r-1}} P_{r-2} \xrightarrow{f_{r-2}} \cdots \xrightarrow{f_2} P_1 \xrightarrow{f_1} P_0 \xrightarrow{f_0} X \rightarrow 0
\rightarrow \cdots \rightarrow 0
\]
becomes exact when we apply the functor $\cX( -,Y )$ for $Y \in \cX$.  Since $\cX$ is a full subcategory, this amounts to the complex becoming exact when we apply the functor $\Hom_{ \Gamma }( -,Y )$ for $Y \in \cX$.  This is true because $\Ext_{ \Gamma }^i( X,Y ) = 0$ for $1 \leqslant i \leqslant d-1$ since $X,Y \in \cX$.
\end{proof}

\begin{Lemma}
\label{split-ext}
Let
\[
  \varepsilon = \left [
  0 \rightarrow X \xrightarrow{f^{-1}} Y^0 \xrightarrow{f^0} \cdots \xrightarrow{f^{d-2}} Y^{d-1} \xrightarrow{f^{d-1}} Z \rightarrow 0\right ]
\]
be a \(d\)-extension in $\modf \Gamma$.
\begin{enumerate}%[label=\emph{\textbf{\roman*}}]
\setlength\itemsep{4pt}

  \item  Suppose $\Ext_{ \Gamma }^i( Y^i,X )=0$ for $1 \leqslant i \leqslant d-1$.  If $\varepsilon$ represents $0$ in $\Ext_{ \Gamma }^d( Z,X )$, then $f^{-1}$ is a split monomorphism. 
  
  \item  Suppose $\Ext_{ \Gamma }^i( Z,Y^{ d-i } )=0$ for $1 \leqslant i \leqslant d-1$.  If $\varepsilon$ represents $0$ in $\Ext_{ \Gamma }^d( Z,X )$, then $f^{d-1}$ is a split epimorphism.
  
\end{enumerate}
\end{Lemma}

\begin{proof}
We show (i) only, (ii) being dual.  Let
\[
  0\rightarrow X\xrightarrow{g^{-1}} I^0\xrightarrow{g^0} I^1\rightarrow \cdots
\]
be an augmented injective resolution. We use it to define the cozysygies \(\sigma^i X\) for \(i\geqslant 0\) which satisfy
\begin{equation}
\label{equ:Ext_vanishing_old}
  \Ext_{ \Gamma }^1(Y^i,\sigma^{i-1}X) =
  \Ext_{ \Gamma }^i(Y^i,X)= 0
\end{equation}
for \(1 \leqslant i \leqslant {d-1}\).  We can construct the following commutative diagram.
\begin{center}
\begin{tikzpicture}[scale=1.4]
\node (0ul) at (0,1) {\(0\)};
\node (0ll) at (0,0) {\(0\)};
\node (Xu) at (1,1) {\(X\)};
\node (Xl) at (1,0) {\(X\)};
\node (dotsu) at (4,1) {\(\cdots\)};
\node (dotsl) at (4,0) {\(\cdots\)};
\node (0ur) at (8,1) {\(0\)};
\node (0lr) at (8,0) {\(0\)};

\node (Y0) at (2,1) {\(Y^0\)};
\node (Y1) at (3,1) {\(Y^1\)};
\node (Yd2) at (5,1) {\(Y^{d-2}\)};
\node (Yd1) at (6,1) {\(Y^{d-1}\)};
\node (Z) at (7,1) {\(Z\)};

\node (I0) at (2,0) {\(I^0\)};
\node (I1) at (3,0) {\(I^1\)};
\node (Id2) at (5,0) {\(I^{d-2}\)};
\node (Id1) at (6,0) {\(I^{d-1}\)};
\node (omega) at (7,0) {\(\sigma^{d} X\)};

%Horisontal arrows
\draw[->] (0ul) 	--  						(Xu);
\draw[->] (0ll) 	-- 						(Xl);
\draw[->] (Xu) 	-- node[above]{\tiny \(f^{-1}\)} 		(Y0);
\draw[->] (Xl) 		-- node[below]{\tiny \(g^{-1}\)}		(I0);
\draw[->] (Y0) 	-- node[above]{\tiny \(f^0\)}  		(Y1);
\draw[->] (I0) 		-- node[below]{\tiny \(g^0\)} 		(I1);
\draw[->] (Y1) 	-- 						(dotsu);
\draw[->] (I1) 		-- 						(dotsl);
\draw[->] (dotsu) 	-- 						(Yd2);
\draw[->] (dotsl) 	-- 						(Id2);
\draw[->] (Yd2) 	-- node[above]{\tiny \(f^{d-2}\)}  	(Yd1);
\draw[->] (Id2) 	-- node[below]{\tiny \(g^{d-2}\)} 	(Id1);
\draw[->] (Yd1) 	-- node[above]{\tiny \(f^{d-1}\)}  		(Z);
\draw[->] (Id1) 	-- node[below]{\tiny \(g\)} 		(omega);
\draw[->] (Z) 		-- 						(0ur);
\draw[->] (omega) 	-- 						(0lr);

%Vertical arrows

\draw[->] (Xu) 	-- node[left]{\tiny \(1_X\)} 		(Xl);
\draw[->] (Y0) 	-- node[left]{\tiny \(h^0\)}		(I0);
\draw[->] (Y1) 	-- node[left]{\tiny \(h^1\)}		(I1);
\draw[->] (Yd2) 	-- node[left]{\tiny \(h^{d-2}\)}	(Id2);
\draw[->] (Yd1) 	-- node[left]{\tiny \(h^{d-1}\)}	(Id1);
\draw[->] (Z) 		-- node[left]{\tiny \(h\)}		(omega);
\end{tikzpicture}
\end{center}
If $\varepsilon$ represents $0$ in $\Ext_{ \Gamma }^d( Z,X )$, then $h$ factors through $g$.  Using Equation \eqref{equ:Ext_vanishing_old} repeatedly, we can then construct the following homotopy.
\begin{center}
\begin{tikzpicture}[scale=1.4]
\node (0ul) at (0,1) {\(0\)};
\node (0ll) at (0,0) {\(0\)};
\node (Xu) at (1,1) {\(X\)};
\node (Xl) at (1,0) {\(X\)};
\node (dotsu) at (4,1) {\(\cdots\)};
\node (dotsl) at (4,0) {\(\cdots\)};
\node (0ur) at (8,1) {\(0\)};
\node (0lr) at (8,0) {\(0\)};

\node (Y0) at (2,1) {\(Y^0\)};
\node (Y1) at (3,1) {\(Y^1\)};
\node (Yd2) at (5,1) {\(Y^{d-2}\)};
\node (Yd1) at (6,1) {\(Y^{d-1}\)};
\node (Z) at (7,1) {\(Z\)};

\node (I0) at (2,0) {\(I^0\)};
\node (I1) at (3,0) {\(I^1\)};
\node (Id2) at (5,0) {\(I^{d-2}\)};
\node (Id1) at (6,0) {\(I^{d-1}\)};
\node (omega) at (7,0) {\(\sigma^{d} X\)};

%Horisontal arrows
\draw[->] (0ul) 	--  						(Xu);
\draw[->] (0ll) 	-- 						(Xl);
\draw[->] (Xu) 	-- node[above]{\tiny \(f^{-1}\)} 		(Y0);
\draw[->] (Xl) 		-- node[below]{\tiny \(g^{-1}\)}		(I0);
\draw[->] (Y0) 	-- node[above]{\tiny \(f^0\)}  		(Y1);
\draw[->] (I0) 		-- node[below]{\tiny \(g^0\)} 		(I1);
\draw[->] (Y1) 	-- 						(dotsu);
\draw[->] (I1) 		-- 						(dotsl);
\draw[->] (dotsu) 	-- 						(Yd2);
\draw[->] (dotsl) 	-- 						(Id2);
\draw[->] (Yd2) 	-- node[above]{\tiny \(f^{d-2}\)}  	(Yd1);
\draw[->] (Id2) 	-- node[below]{\tiny \(g^{d-2}\)} 	(Id1);
\draw[->] (Yd1) 	-- node[above]{\tiny \(f^{d-1}\)}  		(Z);
\draw[->] (Id1) 	-- node[below]{\tiny \(g\)} 		(omega);
\draw[->] (Z) 		-- 						(0ur);
\draw[->] (omega) 	-- 						(0lr);

%Vertical arrows

\draw[->] (Xu) 	-- node[left]{\tiny \(1_X\)} 				(Xl);
\draw[->] (Y0) 	-- node[left, near end]{\tiny \(h^0\)}		(I0);
\draw[->] (Y1) 	-- node[left, near end]{\tiny \(h^1\)}		(I1);
\draw[->] (Yd2) 	-- node[left, near end]{\tiny \(h^{d-\!2}\)}	(Id2);
\draw[->] (Yd1) 	-- node[left, near end]{\tiny \(h^{d-\!1}\)}	(Id1);
\draw[->] (Z) 		-- node[left, near end]{\tiny \(h\)}		(omega);

%Diagonal arrows
\draw[dashed,->] (Y0)		-- node[left, near start]{\tiny \(s^0\)}		(Xl);
\draw[dashed,->] (Y1)		-- node[left, near start]{\tiny \(s^1\)}		(I0);
\draw[dashed,->] (Yd1)	-- node[left, near start]{\tiny \(s^{d-1}\)}	(Id2);
\draw[dashed,->] (Z)		-- node[left, near start]{\tiny \(s\)}		(Id1);
\end{tikzpicture}
\end{center}
Then \(s^0f^{-1}=1_X\) so \(f^{-1}\) is a split monomorphism.
\end{proof}

\section{Lemmas on the functor $\cT( T,- )$}
\label{sec:lemmas2}

The results of this section do not require the full assumptions on $\cT$ made in Setup \ref{set:blanket0}; they are valid if $\cT$ is a $k$-linear $\Hom$-finite category with a Serre functor $S$.

\begin{Lemma}
\label{add-inj}
\begin{enumerate}
\setlength\itemsep{4pt}

  \item  The functor \(\HomT-\) restricts to an equivalence $\add T \rightarrow \proj \Gamma$.

  \item  The functor \(\HomT-\) restricts to an equivalence $\add ST \rightarrow \inj \Gamma$.

\end{enumerate}
\end{Lemma}

\begin{proof}
Part (i) is classic.  For part (ii) note that the Serre functor $S$ gives the following commutative square of functors,
\begin{center}
\begin{tikzpicture}[xscale=4, yscale=2]
\node (aT) at (0,0) {\(\add T\)};
\node (pGo) at (1,0) {\(\proj \Gamma\op\)};
\node (aST) at (0,-1) {\(\add ST\)};
\node (iG) at (1,-1) {\(\inj \Gamma\),};

\draw[->](aT) -- node[above]{\tiny \(\cT(-,T)\)} (pGo);
\draw[->](aT) -- node[left]{\tiny \(S\)} (aST);
\draw[->](aST) -- node[below]{\tiny \(\cT(T,-)\)} (iG);
\draw[->](pGo) -- node[right]{\tiny \(D\)} (iG);
\end{tikzpicture}
\end{center}
where $\proj \Gamma\op$ is the category of projective finite dimensional left $\Gamma$-modules, and the functor $D( - ) = \Hom_k( -,k )$ denotes $k$-linear equivalence.  The functors $S$ and $D$ in the diagram are equivalences, and it is classic that so is $\cT( -,T )$.  Hence the functor $\cT( T,- ): \add ST \rightarrow \inj \Gamma$ is an equivalence.
\end{proof}

\begin{Lemma}
\label{slightlyfaithful}
For \(T'\in \add T\) and \(X\in\T\), the induced maps
\begin{enumerate}
\setlength\itemsep{4pt}

  \item  $\cT( T',X ) 
\rightarrow \Hom_{ \Gamma }( \cT( T,T' ), \cT( T,X ) )$,

  \item  $\cT( X,ST' )
\rightarrow \Hom_{ \Gamma }( \cT( T,X ), \cT( T,ST' ) )$

\end{enumerate}
are bijective.
\end{Lemma}

\begin{proof}
(i): Fixing $X$, the map in (i) is a natural transformation of additive functors of $T' \in \add T$.  Hence it is enough to show bijectivity for $T' = T$, where the map is
\[
  \cT( T,X ) 
  \rightarrow \Hom_{ \Gamma }( \cT( T,T ), \cT( T,X ) )
  = \Hom_{ \Gamma }( \Gamma, \cT( T,X ) ).
\]
This is bijective since it can be identified with the identity map on $\cT( T,X )$.

(ii): The Serre functor $S$ is an autoequivalence so $\Gamma = \cT( ST,ST )$.  An argument analogous to that in (i) shows that the induced map
\begin{equation}
\label{equ:inj_iso1}
  \cT( X,ST' )
\rightarrow \Hom_{ \Gamma\op }( \cT( ST',ST ),\cT( X,ST ) )
\end{equation}
is bijective.  However, there are further bijections
\begin{align}
\nonumber
   \Hom_{ \Gamma\op }( \cT&( ST',ST ),\cT( X,ST ) )
    \rightarrow %\\
%\nonumber
   \Hom_{ \Gamma }( D\cT( X,ST ),D\cT( ST',ST ) )
    \rightarrow \\
\label{equ:inj_iso2}
  & \Hom_{ \Gamma }( \cT( T,X ),\cT( T,ST' ) ),
\end{align}
by $k$-linear and Serre duality.  Using the natural property of the constituent morphisms, it can be checked that the composition of \eqref{equ:inj_iso1} and \eqref{equ:inj_iso2} is the map in (ii) which is hence bijective.
\end{proof}

\begin{Lemma}
\label{lem:objects_sent_to_projectives}
Assume that $\cT( T,- ) : \cT \rightarrow \modf \Gamma$ is a full functor.  If $X \in \cT$ is indecomposable and $\cT( T,X )$ is a projective $\Gamma$-module, then $X \in \add T$.  
\end{Lemma}

\begin{proof}
When $\cT( T,X )$ is projective, Lemma \ref{add-inj}(i) implies that there is some object $T' \in \add T$ such that $\cT( T,T' ) \cong \cT( T,X )$.  Since $\cT( T,- )$ is full, we can find morphisms $T' \xrightarrow{f} X \xrightarrow{g} T'$ which are mapped to inverse isomorphisms by $\cT( T,- )$.  In other words, $\cT( T,gf ) = \cT( T,g )\cT( T,f ) = 1_{ \cT( T,T' ) }$.  

It follows from Lemma \ref{add-inj}(i) that $gf = 1_{ T' }$.  Hence $T'$ is a direct summand of $X$.  But $X$ is indecomposable so in fact $X \cong T' \in \add T$.   
\end{proof}

\begin{Lemma}
\label{lem:weakly_representation_finite}
If $\cT$ has finitely many indecomposable objects, then so does $\cD$.  
\end{Lemma}

\begin{proof}
Since $\cD$ is the essential image of $\cT( T,- )$, each indecomposable object $M \in \cD$ has the form $M \cong \cT( T,X_1 \oplus \cdots \oplus X_n ) \cong \cT( T,X_1 ) \oplus \cdots \oplus \cT( T,X_n )$, where the $X_i$ are indecomposable objects of $\cT$.  Since $M$ is indecomposable, precisely one summand is non-zero, so $M \cong \cT( T,X )$ for an indecomposable object $X \in \cT$.  Since $\cT$ has finitely many indecomposable objects up to isomorphism, it follows that so does $\cD$.
\end{proof}

\begin{Proposition}
\label{funcfin}
Assume that $\cT$ has weak kernels and weak cokernels.  Then $\cD$ is functorially finite in $\modf \Gamma$.
\end{Proposition}

\begin{proof}
Existence of left $\cD$-approximations:  Let \(M\in\modf \Gamma\) have the projective presentation
\[
  \HomT{T_1}\xrightarrow{\HomT f} \HomT{T_0}\xrightarrow{u} M\rightarrow 0,
\]
cf.\ Lemma \ref{add-inj}(i), and let
\[
  T_1 \xrightarrow f T_0 \xrightarrow g X
\]
be a weak cokernel.  Use \(\HomT -\) to get the following commutative diagram in $\modf \Gamma$,
\begin{center}
\begin{tikzpicture}[xscale=3, yscale=2]
\node (TT1u) at (0,1) {\(\HomT{T_1}\)};
\node (TT0u) at (1,1) {\(\HomT{T_0}\)};
\node (TX) at (2,1) {\(\HomT{X},\)};
%\node (dots) at (3.05, 1) {\(\cdots,\)};
\node (TT1d) at (0,2) {\(\HomT{T_1}\)};
\node (TT0d) at (1,2) {\(\HomT{T_0}\)};
\node (M) at (2,2) {\(M\)};
\node (0) at (3, 2) {\(0\)};

\draw[->](TT1u) -- node[below]{\tiny \(\HomT{f}\)} (TT0u);
\draw[->](TT1d) -- node[above]{\tiny \(\HomT{f}\)} (TT0d);
\draw[->](TT0u) -- node[below]{\tiny \(\HomT{g}\)} (TX);
%\draw[->](TX)--(dots);
\draw[->](TT0d) -- node[above]{\tiny $u$} (M);
\draw[->](M)--(0);

\draw[double equal sign distance] (TT1u)--(TT1d);
\draw[double equal sign distance] (TT0u)--(TT0d);
\draw[->] (M)--node[right]{\tiny \(v\)}(TX);
\end{tikzpicture}
\end{center}
where $v$ exists because $M$ is a cokernel while $\cT( T,g ) \circ \cT( T,f ) = \cT( T,gf ) = 0$.  We will show that \(v:M\rightarrow \HomT X\) is a left \(\D\)-approximation of \(M\).

Let \(w:M\rightarrow \HomT Y\) be a homomorphism in $\modf \Gamma$ and consider the composition \(w u:\HomT{T_0}\rightarrow\HomT Y\).  Lemma \ref{slightlyfaithful}(i) gives $wu =\cT( T,h )$ for some \(h:T_0\rightarrow Y\).  Then $hf : T_1 \rightarrow Y$ satisfies $\cT( T,hf ) = wu\cT( T,f ) = w \circ 0 = 0$ whence \(hf=0\) by Lemma \ref{slightlyfaithful}(i).  So \(h\) factors through $g$ and hence \(\cT( T,h ) = wu\) factors through $\cT( T,g ) = vu$.  Since \( u\) is an epimorphism,  this means that \(w\) factors through \(v\) as desired.

Existence of right $\cD$-approximations:  Let $N \in \modf \Gamma$ have the injective copresentation
\[
  0\rightarrow N \rightarrow \HomT{ST^0} \xrightarrow{\cT( T,f )}\HomT{ST^1},
\]
cf.\ Lemma \ref{add-inj}(ii), and let
\[
  Y \rightarrow ST^0 \xrightarrow{f} ST^1
\]
be a weak kernel.  Dually to the above, one shows that there is a right $\cD$-approximation $v: \cT( T,Y ) \rightarrow N$.
\end{proof}

\section{The conditions \ref{condGhostWeak}, \ref{condGhostDualWeak}, \ref{condRes}, \ref{condKer}, \ref{condGhost}, \ref{condGhostDual}}
\label{sec:conditions}

Recall that we still assume Setup \ref{set:blanket0}.  This section introduces the conditions \ref{condGhostWeak}, \ref{condGhostDualWeak}, \ref{condRes}, \ref{condKer}, \ref{condGhost}, \ref{condGhostDual}, and shows how they are linked to cluster tilting in the sense of Oppermann--Thomas.

\begin{Definition}
\label{def:conditions}
The following are conditions which can be imposed on the object $T$:

\begin{VarDescription}{{\textup{(a')}}\quad}
\setlength\itemsep{4pt}

  \myitem[{\textbf{\textup{(a)}}}] \label{condGhostWeak} 
Suppose that $M \in \modf \Gamma$ satisfies $\Ext^j_{ \Gamma }( \cD,M ) = 0$ for $1 \leqslant j \leqslant d-1$, and that \(T_1 \xrightarrow{f} T_0\) is a morphism in \(\add T\) for which
\[
  \HomT{T_1} \xrightarrow {\HomT f} \HomT{T_0} \rightarrow M \rightarrow 0
\]
is a minimal projective presentation in $\modf \Gamma$.  Then there exists a
completion of \(f\) to a $( d+2 )$-angle in $\cT$,
\[
  T_1 \xrightarrow{f} T_0 \xrightarrow{h_{ d+1 }} X_d \xrightarrow{ h_d } \cdots \xrightarrow{h_2} X_1 \xrightarrow{h_1} \Sigma^d T_1,
\]
which satisfies $\cT( T,h_i ) = 0$ for some $1 \leqslant i \leqslant d+1$.

  \myitem[{\textbf{\textup{(a')}}}]\label{condGhostDualWeak} 
Suppose that $N \in \modf \Gamma$ satisfies $\Ext^j_{ \Gamma }( N,\cD ) = 0$ for $1 \leqslant j \leqslant d-1$, and that \(ST_1 \xrightarrow{g} ST_0\) is a morphism in \(\add ST\) for which
\[
  0 \rightarrow N \rightarrow \HomT{ST_1} \xrightarrow {\HomT g} \HomT{ST_0}
\]
is a minimal injective copresentation in $\modf \Gamma$.  Then there exists a
completion of \(g\) to a $( d+2 )$-angle in $\cT$,
\[
  \Sigma^{ -d } ST_0 \xrightarrow{ h_{ d+1 } } X_d \xrightarrow{ h_d } \cdots \xrightarrow{ h_2 } X_1 \xrightarrow{ h_1 } ST_1 \xrightarrow{g} ST_0,
\]
which satisfies $\cT( T,h_i ) = 0$ for some $1 \leqslant i \leqslant d+1$. 

  \myitem[{\textbf{\textup{(b)}}}] \label{condRes}
Suppose that \(X\in \T\) is indecomposable and satisfies \(\HomT X\neq 0\). Then there exists a $( d+2) $-angle in $\cT$,
\[
  T_d\rightarrow \cdots \rightarrow T_0\rightarrow X\xrightarrow{h} \Sigma^d T_d,
\]
with \(T_i \in \add T\) for $0 \leqslant i \leqslant d$, which satisfies \(\HomT h=0\). 

  \myitem[{\textbf{\textup{(c)}}}] \label{condKer}
$\{ X \in \cT \mid \cT( T,X ) = 0 \} = \add\Sigma^d T$.

\end{VarDescription}

Stronger versions of \ref{condGhostWeak} and \ref{condGhostDualWeak} are also useful.

\begin{VarDescription}{{\textup{(strong a')}}\quad}
\setlength\itemsep{4pt}

  \myitem[{\textbf{\textup{(strong a)}}}] \label{condGhost}
The same as condition \ref{condGhostWeak}, except that in the last line we require \(\HomT{ h_d }=0\).

  \myitem[{\textbf{\textup{(strong a')}}}] \label{condGhostDual} 
The same as condition \ref{condGhostDualWeak}, except that in the last line we require \(\HomT{h_2}=0\).
  
\end{VarDescription}
\end{Definition}

Having stated the conditions, the implication (iii) $\Rightarrow$ (ii) in Theorem \ref{thm:C} is clear.  The other implications in the theorem will be proved in Sections \ref{sec:ii_to_i} and \ref{sec:i_to_iii}.

\begin{Lemma}
\label{lem:OT_connection}
$T$ is cluster tilting in the sense of Oppermann--Thomas (see Definition \ref{def:OT}) if and only if it satisfies \ref{condGhostWeak}, \ref{condGhostDualWeak}, \ref{condRes}, and \ref{condKer}. 
\end{Lemma}

\begin{proof}
``If'':  Assume that $T$ satisfies \ref{condGhostWeak}, \ref{condGhostDualWeak}, \ref{condRes}, and \ref{condKer}.  Definition \ref{def:OT}(i) is immediate from \ref{condKer}.

To establish Definition \ref{def:OT}(ii), note that, since the set of $( d+2 )$-angles is closed under direct sums by \cite[def.\ 2.1, (F1)(a)]{GKO}, we can assume that $X \in \cT$ is indecomposable.  If $\cT( T,X ) = 0$ then $X \in \add \Sigma^d T$ by \ref{condKer}, so the trivial $( d+2 )$-angle \[\Sigma^{ -d }X \rightarrow 0 \rightarrow \cdots \rightarrow 0 \rightarrow X \xrightarrow{ 1_X } X\] can be used in Definition \ref{def:OT}(ii).  If $\cT( T,X ) \neq 0$, then the $( d+2 )$-angle from \ref{condRes} can be used in Definition \ref{def:OT}(ii).

``Only if'':  Assume that \(T\) is cluster tilting in the sense of Oppermann--Thomas.

Suppose that we are in the situation of \ref{condGhostWeak}.  Then there is a morphism \(T_1 \xrightarrow{f} T_0\) in \(\add T\) which
we complete to a $( d+2 )$-angle in $\cT$:
\[
  T_1 \xrightarrow{f} T_0 \xrightarrow{ h_{ d+1 } } X_d \xrightarrow{ h_d } \cdots \xrightarrow{ h_2 } X_1 \xrightarrow{ h_1 } \Sigma^d T_1.
\]
Then $\cT( T,h_1 ) = 0$ since $\cT( T,\Sigma^d T ) = 0$, so $T$ satisfies \ref{condGhostWeak}.  A dual argument show that $T$ satisfies \ref{condGhostDualWeak}.

To show that $T$ satisfies \ref{condRes}, we can use the $( d+2 )$-angle from Definition \ref{def:OT}(ii), where $\cT( T,h ) = 0$ since $\cT( T,\Sigma^d T ) = 0$.

To show \ref{condKer}, let $X \in \cT$ be given with $\cT( T,X ) = 0$.  Then $\cT( T_0,X ) = 0$ for each $T_0 \in \add T$.  In particular, the morphism $T_0 \rightarrow X$ in the $( d+2 )$-angle from Definition \ref{def:OT}(ii) is zero, so $h$ is a split monomorphism whence $X \in \add \Sigma^d T$.  Conversely, let $X \in \add \Sigma^d T$ be given.  Then $\cT( T,X ) = 0$ since $\cT( T,\Sigma^d T ) = 0$.
\end{proof}

\section{Proof of the implication (ii) $\Rightarrow$ (i) in Theorem \ref{thm:C}}
\label{sec:ii_to_i}

Recall that we still assume Setup \ref{set:blanket0}.  After providing the necessary ingredients, this section ends with a proof of the implication (ii) $\Rightarrow$ (i) in Theorem \ref{thm:C}.

\begin{Lemma}
\label{lem:resolution}
Let $X \in \cT$ be indecomposable with $\cT( T,X ) \neq 0$ and consider a $( d+2 )$-angle satisfying the requirements in \ref{condRes}.  If we apply the functor $\cT( T,- )$ to all terms but the last, then we get a complex
\[
  \cT( T,T_d ) \rightarrow \cdots \rightarrow \cT( T,T_0 ) \rightarrow \cT( T,X )
\]
which is part of an augmented projective resolution of $\cT( T,X )$ over $\Gamma$.  
\end{Lemma}

\begin{proof}
By \cite[prop.\ 2.5(a)]{GKO}, the complex is exact.  Since $\cT( T,h ) = 0$,  the last morphism is surjective.  By Lemma \ref{add-inj}(i), the $\Gamma$-modules $\cT( T,T_i )$ are projective.
\end{proof}

\begin{Lemma}
\label{lem:rad}
Let $X \in \cT$ be indecomposable with $\cT( T,X ) \neq 0$ and consider a $( d+2 )$-angle satisfying the requirements in \ref{condRes}.  Then $h \in \rad_{ \cT }$.
\end{Lemma}

\begin{proof}
Suppose $h \not\in \rad_{ \cT }$.  If we write $h$ as a matrix $H$ of morphisms from the indecomposable object $X$ to the indecomposable summands of $\Sigma^d T_d$, then one of the entries of $H$ is invertible, say $H_i$.  Let $f : T \rightarrow X$ be a morphism.  Then $hf = 0$ by \ref{condRes}, so in particular $H_if = 0$ whence $f = 0$.  Hence $\cT( T,X ) = 0$, a contradiction.
\end{proof}

\begin{Lemma}
\label{b-direct-summands}
If $T$ satisfies \ref{condRes}, then $\cD$ is closed under direct summands.
\end{Lemma}

\begin{proof}
Consider an object $\cT( T,X )$ of $\cD$.  Suppose $\HomT X = M' \oplus M''$ for some $M',M'' \in \modf \Gamma$.  We will show $M' \in \D$.

Let $X_i$ denote the indecomposable direct summands of $X$.  We can obviously drop each $X_i$ which is mapped to zero by $\cT( T,- )$, so can assume $\cT( T,X_i ) \neq 0$ for each $i$.  Applying \ref{condRes} and Lemma \ref{lem:rad} to each $X_i$ and taking the direct sum of the resulting $( d+2 )$-angles shows that there is a $( d+2 )$-angle 
\[
  T_d \rightarrow \cdots \rightarrow T_0 \xrightarrow g X \xrightarrow h \Sigma^d T_d
\]
with $T_i \in \add T$ for each $i$ and $h \in \rad_{ \cT }$.  

Consider the induced algebra homomorphism
\[
  \pi: \T(X,X)\rightarrow \Hom_\Gamma(\HomT{X},\HomT{X}).
\]
If \(x\in \T(X,X)\) is in the kernel of \(\pi\), then \(xg=0\).  Then $x$ factors through $h$ by \cite[prop.\ 2.5(a)]{GKO} whence $x \in \rad_{ \cT }$.  Hence $\Ker \pi$ is contained in $\rad_{ \cT }( X,X ) = \rad \cT( X,X )$, so idempotents lift through $\pi$ by the combination of 
\cite[cor.\ I.2.3]{bluebook1} and \cite[thm.\ (21.28)]{Lam}.

Hence the projection $e : \cT( T,X ) \rightarrow \cT( T,X )$ onto the direct summand $M'$ can be lifted to an idempotent morphism $f : X \rightarrow X$.  Then $f$ is split by assumption, so $f$ is the projection onto a direct summand $X'$ of $X$, and it follows that $\cT( T,X' ) = M'$ whence $M' \in \cD$.
\end{proof}

\begin{Proposition}
\label{b-full}
If $T$ satisfies \ref{condRes}, then $\cT( T,- ) : \cT \rightarrow \modf \Gamma$ is a full functor.
\end{Proposition}

\begin{proof}

%Goal
Let \(u:\HomT X\rightarrow \HomT Y\) be a morphism in $\modf \Gamma$.  We must find \(f\in \cT(X,Y)\) with \(\HomT f=u\).  Without loss of generality, we can assume that $X$ and $Y$ are indecomposable.

%Resolutions add T to projectives
If \(\HomT X=0\) or \(\HomT Y=0\), then we can set $f = 0$.

If \(\HomT X \neq 0\) and \(\HomT Y \neq 0\), then \ref{condRes} gives two $( d+2 )$-angles in $\cT$,
\begin{align*}
T_d\rightarrow \cdots \rightarrow T_0\xrightarrow{g} X&\xrightarrow{h} \Sigma^d T_d,\\
T'_d\rightarrow \cdots \rightarrow T'_0\xrightarrow{g'} Y&\xrightarrow{h'} \Sigma^d T'_d,
\end{align*}
with $T_i,T_i' \in \add T$ and $\cT( T,h ) = \cT( T,h' ) = 0$.  Applying the functor $\cT( T,- )$ gives the beginning of two augmented projective resolutions by Lemma \ref{lem:resolution}.  Hence the comparison theorem for projective resolutions gives the following commutative diagram. 
\begin{center}
\begin{tikzpicture}[xscale=1.6, yscale=2, font=\small]
\node (TdX) at (0,2) {\(\HomT {T_d}\)};
\node (dotsX) at (1,2) {\(\cdots\)};
\node (T1X) at (2,2) {\(\HomT {T_1}\)};
\node (T0X) at (3.5,2) {\(\HomT {T_0}\)};
\node (X) at (5,2) {\(\HomT X\)};
\node (sTdX) at (6,2) {$0$}; %{\(\HomT {\Sigma^{d}T_d}\)};
\node (TdY) at (0,1) {\(\HomT {T'_d}\)};
\node (dotsY) at (1,1) {\(\cdots\)};
\node (T1Y) at (2,1) {\(\HomT {T'_1}\)};
\node (T0Y) at (3.5,1) {\(\HomT {T'_0}\)};
\node (Y) at (5,1) {\(\HomT {Y}\)};
\node (sTdY) at (6,1) {$0$}; %{\(\HomT {\Sigma^{d}T'_d}\)};

\draw[->] (TdX)--node[right]{\tiny $v_d$}(TdY);
\draw[->] (T1X)--node[right]{\tiny $v_1$}(T1Y);
\draw[->] (T0X)--node[right]{\tiny $v_0$}(T0Y);
\draw[->] (X)-- node[right]{\tiny $u$} (Y);
%\draw[->] (sTdX)--(sTdY);

\draw[->] (TdX)--(dotsX);
\draw[->] (dotsX)--(T1X);
\draw[->] (T1X)--(T0X);
\draw[->] (T0X)--node[above, font=\tiny]{\(\HomT{g}\)}(X);
\draw[->] (X)--(sTdX);

\draw[->] (TdY)--(dotsY);
\draw[->] (dotsY)--(T1Y);
\draw[->] (T1Y)--(T0Y);
\draw[->] (T0Y)--node[below, font=\tiny]{\(\HomT{g'}\)}(Y);
\draw[->] (Y)--(sTdY);
\end{tikzpicture}
\end{center}
By Lemma \ref{add-inj}(i) the second square from the right can be lifted to \(\T\).  Completing to a morphism of $( d+2 )$-angles gives the following commutative diagram.
\begin{center}
\begin{tikzpicture}[xscale=1.6, yscale=2, font=\small]
\node (TdX) at (0,2) {\(T_d\)};
\node (dotsX) at (1,2) {\(\cdots\)};
\node (T1X) at (2,2) {{\(T_1\)}};
\node (T0X) at (3.5,2) {\(T_0\)};
\node (X) at (5,2) {\( X\)};
\node (sTdX) at (6.5,2) {\( \Sigma^{d}T_d\)};
\node (TdY) at (0,1) {\( {T'_d}\)};
\node (dotsY) at (1,1) {\(\cdots\)};
\node (T1Y) at (2,1) {\( {T'_1}\)};
\node (T0Y) at (3.5,1) {\( {T'_0}\)};
\node (Y) at (5,1) {\( {Y}\)};
\node (sTdY) at (6.5,1) {\( {\Sigma^{d}T'_d}\)};

\draw[->] (TdX)--node[right]{\tiny $p_d$}(TdY);
\draw[->] (T1X)--node[right]{\tiny $p_1$}(T1Y);
\draw[->] (T0X)--node[right]{\tiny $p_0$}(T0Y);
\draw[->] (X)-- node[right]{\tiny $f$} (Y);
%\draw[->] (sTdX)--(sTdY);

\draw[->] (TdX)--(dotsX);
\draw[->] (dotsX)--(T1X);
\draw[->] (T1X)--(T0X);
\draw[->] (T0X)--node[above]{\tiny $g$}(X);
\draw[->] (X)--(sTdX);

\draw[->] (TdY)--(dotsY);
\draw[->] (dotsY)--(T1Y);
\draw[->] (T1Y)--(T0Y);
\draw[->] (T0Y)--node[below]{\tiny $g'$}(Y);
\draw[->] (Y)--(sTdY);
\end{tikzpicture}
\end{center}
The first diagram gives
\[
  u\cT( T,g ) = \cT( T,g' )v_0 = (*).
\]
We know $v_0 = \cT( T,p_0 )$ so have
\[
  (*) = \cT( T,g' )\cT( T,p_0 ) = \cT( T,f )\cT( T,g ),
\]
where the last equality is by the second diagram.  Since $\cT( T,g )$ is surjective, it follows that $u = \cT( T,f )$.
\end{proof}

\begin{Proposition}
\label{b-d-rigid}
If $T$ satisfies \ref{condRes}, then $\cD$ is a $d$-rigid subcategory of $\modf \Gamma$, that is, $\Ext_{ \Gamma }^i( \cD,\cD ) = 0$ for $1 \leqslant i \leqslant d-1$.
\end{Proposition}

\begin{proof}
It is enough to see that if $X,Y \in \cT$ are indecomposable, then
\begin{equation}
\label{equ:Ext_vanishing}
  \Ext_{ \Gamma }^i( \cT( T,X ),\cT( T,Y )) = 0
  \;\;\;\mbox{for}\;\;\;
  1 \leqslant i \leqslant d-1.
\end{equation}
This is clear for $\cT( T,X ) = 0$, so we can assume $\cT( T,X ) \neq 0$.  Condition \ref{condRes} gives a $( d+2 )$-angle $T_d\rightarrow \cdots \rightarrow T_0\rightarrow X\xrightarrow{h} \Sigma^d T_d$, and Lemma \ref{lem:resolution} implies that
\[
  \cT( T,T_d ) \rightarrow \cdots \cT( T,T_0 )
\]
are the first $d+1$ terms of a projective resolution of $\cT( T,X )$.  Hence the homology groups of the complex
\begin{equation}
\label{equ:complex_for_Ext}
  \Hom_{ \Gamma }( \cT( T,T_0 ),\cT( T,Y ))
  \rightarrow \cdots \rightarrow
  \Hom_{ \Gamma }( \cT( T,T_d ),\cT( T,Y ))
\end{equation}
are the $\Ext$ groups in Equation \eqref{equ:Ext_vanishing}.  But Lemma \ref{slightlyfaithful}(i) says that \eqref{equ:complex_for_Ext} is isomorphic to
\[
  \cT( T_0,Y ) \rightarrow \cdots \rightarrow \cT( T_d,Y )
\]
which is exact by \cite[prop.\ 2.5(a)]{GKO}.  Hence Equation \eqref{equ:Ext_vanishing} is satisfied.
\end{proof}

\begin{Proposition}
\label{a-Ext-closed}
\begin{enumerate}
\setlength\itemsep{4pt}

  \item  Assume \(T\) satisfies \ref{condGhostWeak} and \ref{condRes}.  
  
  If \(M\in \modf\Gamma\) satisfies \(\Ext_{\Gamma}^i( \cD,M )=0\) for $1 \leqslant i \leqslant d-1$, then \(M\in \D\).

  \item  Assume \(T\) satisfies \ref{condGhostDualWeak} and \ref{condRes}.  
  
  If \(N\in \modf\Gamma\) satisfies \(\Ext_{\Gamma}^i( N,\cD )=0\) for $1 \leqslant i \leqslant d-1$, then \(N\in \D\).

\end{enumerate}
\end{Proposition}

\begin{proof}
(i):  By Lemma \ref{add-inj}(i) we can pick a morphism \(T_1 \xrightarrow{f} T_0\) in \(\add T\) such that
\[
  \HomT{T_1} \xrightarrow{ \cT( T,f ) } \HomT{T_0} \rightarrow M \rightarrow 0
\]
is a minimal projective presentation in $\modf \Gamma$.  By \ref{condGhostWeak} there exists a $( d+2 )$-angle in $\cT$,
\[
  T_1 \xrightarrow{f} T_0 \xrightarrow{h_{ d+1 }} X_d \xrightarrow{ h_d } \cdots \xrightarrow{h_2} X_1 \xrightarrow{h_1} \Sigma^d T_1,
\]
such that $\cT( T,h_i ) = 0$ for some $1 \leqslant i \leqslant d+1$.  There is an induced long exact sequence
\begin{align*}
  \cT( T,T_1 ) \xrightarrow{ \cT( T,f ) } \cT( T,T_0 ) \xrightarrow{ \cT( T,h_{ d+1 } ) } \cT( T,X_d ) \xrightarrow{ \cT( T,h_d ) } \cdots \\
\cdots  \xrightarrow{ \cT( T,h_2 ) } \cT( T,X_1 ) \xrightarrow{ \cT( T,h_1 ) } \cT( T,\Sigma^d T_1 ).
  \end{align*}

If $\cT( T,h_{ d+1 } ) = 0$ then $\cT( T,f )$ is surjective whence $M = 0$ so $M \in \cD$.

If $\cT( T,h_d ) = 0$ then $M \cong \cT( T,X_d )$ so $M \in \cD$.

If $\cT( T,h_i ) = 0$ for some $1 \leqslant i \leqslant d-1$, then the long exact sequence induces an exact sequence
\[
  0 \rightarrow M \xrightarrow{ \mu } \cT( T,X_d ) \rightarrow \cdots \rightarrow \cT( T,X_i ) \rightarrow 0.
\]
This is a $( d-i )$-extension representing an element in $\Ext^{ d-i }_{ \Gamma }( \cT( T,X_i ),M )$.  This $\Ext$ is zero by the assumption on $M$.    
It follows from Lemma \ref{split-ext}(i) that $\mu$ is split injective.  So $M$ is a direct summand of $\cT( T,X_d )$ which is in $\cD$, so $M \in \cD$ by Lemma \ref{b-direct-summands}.

(ii):  This is proved dually to (i).
\end{proof}

{\noindent \it Proof} of Theorem \ref{thm:C}, the implication (ii) $\Rightarrow$ (i):  Under condition (ii) in Theorem \ref{thm:C}, the functor $\cT( T,- ) : \cT \rightarrow \modf \Gamma$ is full by Proposition \ref{b-full}, and its essential image $\cD$ is $d$-cluster tilting in $\modf \Gamma$ by Propositions \ref{funcfin}, \ref{b-d-rigid}, and \ref{a-Ext-closed}.
\hfill $\Box$

\section{Proof of the implication (i) $\Rightarrow$ (iii) in Theorem \ref{thm:C}}
\label{sec:i_to_iii}

Recall that we still assume Setup \ref{set:blanket0}.  After providing the necessary ingredients, this section ends with a proof of the implication (i) $\Rightarrow$ (iii) in Theorem \ref{thm:C}.

\begin{Lemma}
\label{diagram-lemma-split-mono}
Assume that $\cT( T,- ) : \cT \rightarrow \modf \Gamma$ is a full functor, and that we are given the following commutative diagram in $\cT$.
\begin{center}
\begin{tikzpicture}[scale=2]
\node(T1) at (0,1) {\(T_1\)};
\node(T0) at (1,1) {\(T_0\)};
\node(T1') at (0,0) {\(T_1'\)};
\node(T0') at (1,0) {\(T_0'\)};
\node(X) at (2,1) {\(X\)};
\node(Z) at (2.0,0) {\(Z\)};
\draw[->] (T1) -- node[above]{\tiny \(f\)}(T0);
\draw[->] (T1') -- node[below]{\tiny \(f'\)}(T0');
\draw[->] (T0) -- node[above]{\tiny \(g\)}(X);
\draw[->] (T0') -- node[below]{\tiny \(g'\)}(Z);
\draw[->] (T1) -- node[left]{\tiny \(h_1\)}(T1');
\draw[->] (T0) -- node[right]{\tiny \(h_0\)}(T0');
\end{tikzpicture}
\end{center}

Suppose the following are satisfied:
\begin{enumerate}
\setlength\itemsep{4pt}

  \item $T_0, T_1' \in \add T$.

  \item $h_1$ and $h_0$ are isomorphisms.

  \item The first row is part of a $( d+2 )$-angle in $\cT$, with $g$ left minimal.

  \item $\cT( T,g' )$ is a weak cokernel of $\cT( T,f' )$ in $\cD$.

\end{enumerate}
Then there exists a split monomorphism $v : X \rightarrow Z$ completing to a larger commutative diagram.  

Suppose we also have:
\begin{enumerate}
\setcounter{enumi}{4}
\setlength\itemsep{4pt}

  \item \(g'\) is left minimal.

\end{enumerate}
Then $v$ is an isomorphism.
\end{Lemma}\label{CokernelToSplit}

\begin{proof}
Condition (iv) implies $\cT( T,g'f' ) = \cT( T,g' )\cT( T,f' ) = 0$ when\-ce $g'f' = 0$ by condition (i) and Lemma \ref{slightlyfaithful}(i).  Hence $g'h_0f = g'f'h_1 = 0$ and it follows from condition (iii) and \cite[prop.\ 2.5(a)]{GKO} that there is a morphism $v : X \rightarrow Z$ such that
\begin{equation}
\label{equ:ghvg}
  g'h_0 = vg.
\end{equation}
We will show that \(v\) is a split monomorphism.

Condition (ii) says that $f'$ and $f$ are isomorphic in the morphism category of $\cT$, so hence $\cT( T,f' )$ and $\cT( T,f )$ are isomorphic in the morphism category of $\modf \Gamma$.  
Since $\cT( T,f' )$ has weak cokernel $\cT( T,g' )$, it follows that $\cT( T,f )$ has weak cokernel $\cT( T,g' )\cT( T,h_0 ) = \cT( T,g'h_0 )$.  
Then $\cT( T,g )\cT( T,f ) = \cT( T,gf ) = 0$ implies that there is $\phi : \cT( T,Z ) \rightarrow \cT( T,X )$ such that $\cT( T,g ) = \phi \circ \cT( T,g'h_0 )$.  
Since $\cT( T,- )$ is full, there exists some $w : Z \rightarrow X$ such that $\cT( T,w ) = \phi$, and it follows that $\cT( T,g ) = \cT( T,wg'h_0 )$.  
By condition (i) and Lemma \ref{slightlyfaithful}(i) this implies 
\begin{equation}
\label{equ:gwgh}
  g = wg'h_0.
\end{equation}

Equations \eqref{equ:ghvg} and \eqref{equ:gwgh} imply $g = wvg$.  Condition (iii) says that $g$ is left minimal, so $wv$ is an isomorphism.  In particular, $v$ is a split monomorphism.

Now suppose that $g'$ is left minimal.  Then so is $g'h_0$ since $h_0$ is an isomorphism by condition (ii).  Equations \eqref{equ:ghvg} and \eqref{equ:gwgh} imply $vwg'h_0 = g'h_0$, so $vw$ is an isomorphism.  We already proved that so is $wv$, so $v$ is an isomorphism.
\end{proof}

\begin{Proposition}
\label{aFulfilled}
If $\cT( T,- ) : \cT \rightarrow \modf \Gamma$ is a full functor and $\cD$ is a $d$-cluster tilting subcategory of $\modf \Gamma$, then $T$ satisfies \ref{condGhost} and \ref{condGhostDual}.
\end{Proposition}

\begin{proof}
Suppose that we are in the situation of \ref{condGhost}, that is, $M \in \modf \Gamma$ satisfies
\begin{equation}
\label{equ:Ext_vanishing2}
  \Ext^j_{ \Gamma }( \cD,M ) = 0
  \;\;\;\mbox{for}\;\;\;
  1 \leqslant j \leqslant d-1,
\end{equation}
and \(T_1 \xrightarrow{f} T_0\) is a morphism in \(\add T\) for which
\begin{equation}
\label{equ:projective_presentation}
  \HomT{T_1} \xrightarrow {\HomT f} \HomT{T_0} \rightarrow M \rightarrow 0
\end{equation}
is a minimal projective presentation in $\modf \Gamma$.

Complete $f$ to a $( d+2 )$-angle
\[
  T_1 \xrightarrow{f} T_0 \xrightarrow{h_{ d+1 }} X_d \xrightarrow{ h_d } \cdots \xrightarrow{h_2} X_1 \xrightarrow{h_1} \Sigma^d T_1,
\]
with $h_i \in \rad_{ \cT }$ for $2 \leqslant i \leqslant d$, 
see \cite[lem.\ 5.18(2)]{OppermannT}.  Then $h_i$ is left minimal for $3 \leqslant i \leqslant d+1$ by \cite[lem.\ 2.11]{F}.  We will show $\cT( T,h_d ) = 0$, thereby establishing \ref{condGhost}.

Since $\cD$ is a $d$-cluster tilting subcategory, Equation \eqref{equ:Ext_vanishing2} implies $M \in \cD$.  Hence there is $Z \in \cT$ with $M \cong \cT( T,Z )$.  Since $\cT( T,- )$ is full, there is a diagram
\[
  T_1 \xrightarrow{f} T_0 \xrightarrow{g} Z
\]
which $\cT( T,- )$ maps to \eqref{equ:projective_presentation}.  We get a diagram,
\begin{center}
\begin{tikzpicture}[scale=2]
\node(T1) at (0,1) {\(T_1\)};
\node(T0) at (1,1) {\(T_0\)};
\node(T1') at (0,0) {\(T_1\)};
\node(T0') at (1,0) {\(T_0\)};
\node(X) at (2,1) {\(X_d\)};
\node(Z) at (2.05,0) {\(Z\),};
\draw[->] (T1) -- node[above]{\tiny \(f\)}(T0);
\draw[->] (T1') -- node[below]{\tiny \(f\)}(T0');
\draw[->] (T0) -- node[above]{\tiny \(h_{ d+1 }\)}(X);
\draw[->] (T0') -- node[below]{\tiny \(g\)}(Z);
\draw[double equal sign distance] (T1) -- (T1');
\draw[double equal sign distance] (T0) -- (T0');
\end{tikzpicture}
\end{center}
which satisfies conditions (i)--(iv) in Lemma \ref{diagram-lemma-split-mono}.  The lemma provides a split monomorphism $v : X_d \rightarrow Z$ satisfying $vh_{ d+1 } = g$.

To show $\cT( T,h_d ) = 0$, let $a : T \rightarrow X_d$ be given.  Consider $va : T \rightarrow Z$.  Since $\cT( T,g )$ is surjective, there is $u : T \rightarrow T_0$ such that $gu = va$.  Thus $vh_{ d+1 }u = gu = va$ whence $h_{ d+1 }u = a$ because $v$ is a split monomorphism.  But then $h_d a = h_d h_{ d+1 }u = 0 \circ u = 0$ as desired, where we used \cite[prop.\ 2.5(a)]{GKO}.

Condition \ref{condGhostDual} is established by a dual argument.
\end{proof}

\begin{Proposition}
\label{bFulfilled}
If $\cT( T,- ) : \cT \rightarrow \modf \Gamma$ is a full functor and $\cD$ is a $d$-cluster tilting subcategory of $\modf \Gamma$, then $T$ satisfies \ref{condRes}.
\end{Proposition}

\begin{proof}
Let \(X\in \T\) be indecomposable with \(\HomT X\neq 0\).

If $\cT( T,X )$ is a projective $\Gamma$-module, then \(X\in\add T\) by Lemma \ref{lem:objects_sent_to_projectives}, so the trivial $( d+2 )$-angle $0 \rightarrow \cdots \rightarrow 0 \rightarrow X \xrightarrow{ 1_X } X \rightarrow \Sigma^d 0$ can be used in \ref{condRes}.

Suppose that $\cT( T,X )$ is not a projective $\Gamma$-module.  By Lemmas \ref{add-inj}(i) and \ref{slightlyfaithful}(i), the augmented minimal projective resolution of $\cT( T,X )$ can be written in the form
\[
  \cdots \xrightarrow{\HomT{f_{d+1}}} \HomT{T_d} \xrightarrow{\HomT{f_{d}}} \cdots \xrightarrow{\HomT{f_1}} \HomT{T_0} \xrightarrow{\HomT{f_0}} \HomT{X}\rightarrow 0
\]
with the $T_i$ in $\add T$.  The morphism \(f_d\) can be completed to a $( d+2 )$-angle, which is the first row in the following diagram.
\begin{center}
\begin{tikzpicture}[scale=1.5]
\node (Tdupper) at (0,1) {\(T_d\)};
\node (Td-upper) at (1,1) {\(T_{d-1}\)};
\node (Tdlower) at (0,0) {\(T_d\)};
\node (Td-lower) at (1,0) {\(T_{d-1}\)};
\node (X-2) at (2,1) {\(X_{d-2}\)};
\node (dotsu) at (3,1) {\(\cdots\)};
\node (X1) at (4,1){\(X_1\)};
\node (X0) at (5,1){\(X_0\)};
\node (Y) at (6,1){\(Y\)};
\node (STd) at (7,1){\(\Sigma^d T_d\)};
%\node () at (,0){\(\)};

\node (T-2) at (2,0){\(T_{d-2}\)};
\node (dotsl) at (3,0){\(\cdots\)};
\node (T1) at (4,0){\(T_1\)};
\node (T0) at (5,0){\(T_0\)};
\node (X) at (6,0){\(X\)};

\draw[->] (Tdupper)--node[above]{\tiny \(f_d\)}(Td-upper);
\draw[->] (Td-upper)--node[above]{\tiny \(g_{d-1}\)}(X-2);
\draw[->] (X-2) -- node[above]{\tiny \(g_{d-2}\)}(dotsu);
\draw[->] (dotsu) -- node[above]{\tiny \(g_{2}\)}(X1);
\draw[->] (X1) -- node[above]{\tiny \(g_1\)}(X0);
\draw[->] (X0) -- node[above]{\tiny \(g_0\)}(Y);
\draw[->] (Y) -- node[above]{\tiny \(h\)}(STd);

\draw[->] (Tdlower)--node[below]{\tiny \(f_d\)}(Td-lower);
\draw[->] (Td-lower) -- node[below]{\tiny \(f_{d-1}\)}(T-2);
\draw[->] (T-2) -- node[below]{\tiny \(f_{d-2}\)}(dotsl);
\draw[->] (dotsl) -- node[below]{\tiny \(f_2\)}(T1);
\draw[->] (T1) -- node[below]{\tiny \(f_1\)}(T0);
\draw[->] (T0) -- node[below]{\tiny \(f_0\)}(X);

\draw [double equal sign distance] (Tdupper)--(Tdlower);
\draw [double equal sign distance] (Td-upper)--(Td-lower);
%\draw[->] () -- node[above]{\(\)}();
\end{tikzpicture}
\end{center}
We will use Lemma \ref{diagram-lemma-split-mono} repeatedly.  We start by verifying conditions (i)--(v) in the lemma for some of the objects and morphisms in the diagram.

(i): We already know that the $T_i$ are in $\add T$.

(ii): The identity morphisms in the diagram are isomorphisms.

(iii): When constructing the $( d+2 )$-angle in the first row of the diagram, we can assume $g_i \in \rad_{ \cT }$ for $0 \leqslant i \leqslant d-2$ by \cite[lem.\ 5.18(2)]{OppermannT}.  Then $g_i$ is left minimal for $1 \leqslant i \leqslant d-1$ by \cite[lem.\ 2.11]{F}.  Moreover, $\cT( T,f_d )$ is a morphism in a minimal projective resolution so is right minimal.  By Lemma \ref{add-inj}(i), so is $f_d$.  Then $\Sigma^d f_d$ is right minimal, forcing $h \in \rad_{ \cT }$ and hence $g_0$ left minimal by \cite[lem.\ 2.11]{F}.  Summing up, $g_i$ is left minimal for $0 \leqslant i \leqslant d-1$.

(iv): The complex
\[
  \HomT{T_{d-1}} \xrightarrow{\HomT{f_{d-1}}} \cdots \xrightarrow{\HomT{f_1}} \HomT{T_0} \xrightarrow{\HomT{f_0}} \HomT{X} \rightarrow 0
\]
is a \(d\)-cokernel of \(\HomT{f_{d}}\) by Lemma \ref{projres-d-cokernel}. 
In particular, \(\HomT{f_i}\) is a weak cokernel of \(\HomT{f_{i+1}}\) for $1 \leqslant i \leqslant d-1$, and $\cT( T,f_0 )$ is a cokernel of $\cT( T,f_1 )$.

(v):  Lemma \ref{left-min-proj-res} says that $\cT( T,f_i )$ is left minimal for $1 \leqslant i \leqslant d$.  By Lemma \ref{add-inj}(i) this implies that $f_i$ is left minimal for $1 \leqslant i \leqslant d$.

We can now use Lemma \ref{diagram-lemma-split-mono} repeatedly to get the following commutative diagram.
\begin{center}
\begin{tikzpicture}[scale=1.5]
\node (Tdupper) at (0,1) {\(T_d\)};
\node (Td-upper) at (1,1) {\(T_{d-1}\)};
\node (Tdlower) at (0,0) {\(T_d\)};
\node (Td-lower) at (1,0) {\(T_{d-1}\)};
\node (X-2) at (2,1) {\(X_{d-2}\)};
\node (dotsu) at (3,1) {\(\cdots\)};
\node (X1) at (4,1){\(X_1\)};
\node (X0) at (5,1){\(X_0\)};
\node (Y) at (6,1){\(Y\)};
\node (STd) at (7,1){\(\Sigma^d T_d\)};
%\node () at (,0){\(\)};

\node (T-2) at (2,0){\(T_{d-2}\)};
\node (dotsl) at (3,0){\(\cdots\)};
\node (T1) at (4,0){\(T_1\)};
\node (T0) at (5,0){\(T_0\)};
\node (X) at (6,0){\(X\)};

\draw[->] (Tdupper)--node[above]{\tiny \(f_d\)}(Td-upper);
\draw[->] (Td-upper)--node[above]{\tiny \(g_{d-1}\)}(X-2);
\draw[->] (X-2) -- node[above]{\tiny \(g_{d-2}\)}(dotsu);
\draw[->] (dotsu) -- node[above]{\tiny \(g_{2}\)}(X1);
\draw[->] (X1) -- node[above]{\tiny \(g_1\)}(X0);
\draw[->] (X0) -- node[above]{\tiny \(g_0\)}(Y);
\draw[->] (Y) -- node[above]{\tiny \(h\)}(STd);

\draw[->] (Tdlower)--node[below]{\tiny \(f_d\)}(Td-lower);
\draw[->] (Td-lower) -- node[below]{\tiny \(f_{d-1}\)}(T-2);
\draw[->] (T-2) -- node[below]{\tiny \(f_{d-2}\)}(dotsl);
\draw[->] (dotsl) -- node[below]{\tiny \(f_2\)}(T1);
\draw[->] (T1) -- node[below]{\tiny \(f_1\)}(T0);
\draw[->] (T0) -- node[below]{\tiny \(f_0\)}(X);

\draw [double equal sign distance] (Tdupper)--(Tdlower);
\draw [double equal sign distance] (Td-upper)--(Td-lower);
%\draw[->] () -- node[above]{\(\)}();

\draw[->] (X-2) -- node[sloped,below]{\tiny $\cong$} (T-2);
\draw[->] (X1) -- node[sloped,below]{\tiny $\cong$} (T1);
\draw[->] (X0) -- node[sloped,below]{\tiny $\cong$} (T0);
\draw[right hook->] (Y) -- node[left]{\tiny $\upsilon$} node[right]{\tiny split} (X);
\end{tikzpicture}
\end{center}
In the final step, we only know that $g_0$ is left minimal, not that $f_0$ is left minimal, so Lemma \ref{diagram-lemma-split-mono} only gives a split monomorphism $\upsilon : Y \hookrightarrow X$.  However, $X$ is indecomposable so $\upsilon$ is either zero or an isomorphism.  If it were zero, then the rightmost commutative square in the diagram would force $f_0 = 0$ whence $\cT( T,f_0 ) = 0$, contradicting that $\cT( T,f_0 )$ is a surjection onto the non-zero module $\cT( T,X )$.  It follows that $\upsilon$ is an isomorphism.

Hence the rightmost commutative square in the diagram implies that $\cT( T,g_0 )$ is surjective, and so $\cT( T,h ) = 0$ by \cite[prop.\ 2.5(a)]{GKO}.  Hence the $( d+2 )$-angle in the first row of the diagram can be used in \ref{condRes}.
\end{proof}

{\noindent \it Proof} of Theorem \ref{thm:C}, the implication (i) $\Rightarrow$ (iii):
Assuming condition (i) in Theorem \ref{thm:C}, the object $T$ satisfies \ref{condGhost} and \ref{condGhostDual} by Proposition \ref{aFulfilled}, and \ref{condRes} by Proposition \ref{bFulfilled}.
\hfill $\Box$

\section{Proof of Theorem \ref{thm:AB} and Corollary \ref{cor:repfin}}
\label{sec:proofs}

Recall that we still assume Setup \ref{set:blanket0}.

{\medskip \noindent \it Proof} of Theorem \ref{thm:AB}:  

(i):  This follows from Theorem \ref{thm:C} and Lemma \ref{lem:OT_connection}.

(ii):  Theorem \ref{thm:C} and Lemma \ref{lem:OT_connection} show that $\cT( T,- ) : \cT \rightarrow \modf \Gamma$ is full, so (ii) amounts to showing that if $g : X \rightarrow Y$ is a morphism in $\cT$, then
\[
  \cT( T,g ) = 0
  \Leftrightarrow
  \mbox{$g$ factors through an object in $\add \Sigma^d T$.}
\]

To show $\Rightarrow$, consider the $( d+2 )$-angle in Definition \ref{def:OT}(ii).  Since $T_0 \in \add T$, the condition $\cT( T,g ) = 0$ implies $gf_0 = 0$.  By \cite[prop.\ 2.5(a)]{GKO} there is $g' : \Sigma^d T_d \rightarrow Y$ such that $g = g'h$.  That is, $g$ has been factored through $\Sigma^d T_d \in \add \Sigma^d T$.  The implication $\Leftarrow$ is clear since $\cT( T,\Sigma^d T ) = 0$ by Definition \ref{def:OT}(i). 

(iii):  Recall that $S$ is the Serre functor of $\cT$.  By Definition \ref{def:OT}(ii) there is a $( d+2 )$-angle in $\cT$,
\[
  T_d  \xrightarrow{}\cdots \xrightarrow{} T_0 \xrightarrow{} ST \xrightarrow{} \Sigma^dT_d,
\]
with the $T_i$ in $\add T$.  Applying $\cT( T,- )$ gives a sequence in $\modf \Gamma$,
\[
  \HomT{\Sigma^{-d}ST} \rightarrow\HomT{T_d} \rightarrow{\cdots}\rightarrow \HomT{T_0} \rightarrow \cT( T,ST ) \rightarrow \HomT{\Sigma^dT_d},
\]
which is exact by \cite[prop.\ 2.5(a)]{GKO}.  By Serre duality we have \[\cT( T,\Sigma^{ -d }ST ) \cong D\cT( T,\Sigma^d T ) = 0\text{ and }\cT( T,ST ) \cong D\cT( T,T ) = D\Gamma\] as right-$\Gamma$-modules.  Moreover, $\cT( T,\Sigma^d T_d ) = 0$ by Definition \ref{def:OT}(i).  The sequence hence reads
\[
  0 \rightarrow\HomT{T_d} \rightarrow{\cdots}\rightarrow \HomT{T_0} \rightarrow ( D\Gamma )_{ \Gamma } \rightarrow 0.
\]
This provides a projective resolution of $( D\Gamma )_{ \Gamma }$ with at most $d+1$ non-zero projective modules.  Consequently, each injective right $\Gamma$-module has projective dimension $\leqslant d$.  

The opposite category $\cT^{ \op }$ is $( d+2 )$-angulated, and $T$ is a cluster tilting object in the sense of Oppermann--Thomas of $\cT^{ \op }$ with endomorphism algebra $\Ext_{ \cT^{ \op } }T = \Gamma^{ \op }$.  Applying to this setup what we already proved shows that each injective right $\Gamma^{ \op }$-module has projective dimension $\leqslant d$.  That is, each injective left $\Gamma$-module has projective dimension $\leqslant d$.

The statements about the injective dimension of projective modules follow by $k$-linear duality.

(iv):  It is well-known that if the global dimension of $\Gamma$ is finite, then it is equal to the projective dimension of $( D\Gamma )_{ \Gamma }$, so (iv) follows from (iii).
\hfill $\Box$

{\medskip \noindent \it Proof} of Corollary \ref{cor:repfin}:  

(i):  Lemma \ref{lem:weakly_representation_finite} says that $\cD$ has finitely many indecomposable objects, so $\cD = \add M$ for some $M \in \modf \Gamma$.  Theorem \ref{thm:AB}(i) says that $\cD$ is a $d$-cluster tilting subcategory of $\modf \Gamma$, so $M$ is a $d$-cluster tilting module.  Hence $\Gamma$ is weakly $d$-representation finite.

(ii):  If $\Gamma$ has finite global dimension, then the global dimension is at most $d$ by Theorem \ref{thm:AB}(iv).  Since $\Gamma$ is weakly $d$-representation finite by (i), it is then $d$-representation finite.
\hfill $\Box$

\section{Two classes of examples}
\label{sec:examples}

We conclude with two classes of examples.  The first shows how Theorem \ref{thm:AB} and Corollary \ref{cor:repfin} imply \cite[thm.\ 5.6]{OppermannT}.  Recall that the notion of $d$-representation finite algebras was defined in \cite[def.\ 2]{IO}, and that large classes of such algebras exist, see for instance \cite[thm.\ 3.11]{HI} and \cite[sec.\ 5]{IO}.

\begin{Example}
\label{exa:OT}
Let $\Lambda$ be a $d$-representation finite $k$-algebra.  In \cite[sec.\ 5]{OppermannT} was constructed a so-called $( d+2 )$-angulated cluster category $\cO_{ \Lambda }$.  Let $T \in \cO_{ \Lambda }$ be a cluster tilting object in the sense of Oppermann--Thomas with endomorphism algebra $\Gamma = \End_{ \cO_{ \Lambda } }T$.

Our results apply to this situation because the conditions of Setup \ref{set:blanket0} are satisfied:  The category $\cO_{ \Lambda }$ is $k$-linear $\Hom$-finite with split idempotents by construction, see \cite[thms.\ 5.14 and 5.25]{OppermannT}, and it has a Serre functor by \cite[thm.\ 5.2]{OppermannT}.  Observe that $\cO_{ \Lambda }$ has finitely many indecomposable objects by \cite[thm.\ 5.2(1)]{OppermannT}.

We can recover the results of \cite[thm.\ 5.6]{OppermannT} on $\cO_{ \Lambda }$, $T$, and $\Gamma$ as follows:  Consider the functor $\cO_{ \Lambda }( T,- ) : \cO_{ \Lambda } \rightarrow \modf \Gamma$.  Theorem \ref{thm:AB} says that its essential image $\cD$ is $d$-cluster tilting in $\modf \Gamma$, that $\cO_{ \Lambda }( T,- )$ induces an equivalence of categories
\[
  \cO_{ \Lambda } / [ \add \Sigma^d T ] \xrightarrow{ \sim } \cD,
\]
and that $\Gamma$ is $d$-Gorenstein.  Corollary \ref{cor:repfin} says that $\Gamma$ is weakly $d$-representation finite, and that if it has finite global dimension, then it is $d$-representation finite.
\end{Example}

\begin{Example}
Let $\Lambda$ be a $d$-representation finite $k$-algebra.  Let $\cF$ be the unique $d$-cluster tilting subcategory of $\modf \Lambda$, and consider the full subcategory
\[
  \overline{ \cF }
  =
  \add \{ \Sigma^{ id }F \mid i \in \BZ, F \in \cF \}
\]
of the derived category $\Db( \modf \Lambda )$.  It is clearly invariant under $\Sigma^d$, and it is a $d$-cluster tilting subcategory of $\Db( \modf \Lambda )$ by \cite[thm.\ 1.21]{Iyama}, so it is a $( d+2 )$-angulated category by \cite[thm.\ 1]{GKO}.

Our results apply to this situation because the conditions of Setup \ref{set:blanket0} are satisfied:  The category $\overline{ \cF }$ is $k$-linear $\Hom$-finite with split idempotents because it is a full subcategory of $\Db( \modf \Lambda )$ closed under direct summands, and \cite[thm.\ 3.1]{IO2} implies that the Serre functor $S_D$ of $\Db( \modf \Lambda )$ restricts to a Serre functor $S$ of $\overline{ \cF }$.

Set $T = \Lambda_{ \Lambda }$.  Then $\overline{ \cF }( T,- )$ is a restriction of $\Hom_{ \Db( \modf \Lambda ) }( \Lambda_{ \Lambda },- )$, so the endomorphism algebra $\End_{ \overline{ \cF } }T$ is
\[
  \overline{ \cF }( T,T ) \cong \Lambda
\]
and the functor $\overline{ \cF }( T,- ) : \overline{ \cF } \rightarrow \modf \Lambda$ can be identified with
\[
  \H^0 : \overline{ \cF } \rightarrow \modf \Lambda.
\]

By definition, each object of $\overline{ \cF }$ is a finite direct sum of the form
\[
  \overline{ F } = \bigoplus_i \Sigma^{ id }F_i,
\]
and $\H^0( \overline{ F } ) = F_0$.  It follows that the essential image of $\overline{ \cF }( T,- ) = \H^0( - )$ is $\cF$, which is $d$-cluster tilting in $\modf \Lambda$, and that $\overline{ \cF }( T,- ) = \H^0( - )$ is full.

By Theorem \ref{thm:C} the object $T$ satisfies \ref{condGhost}, \ref{condGhostDual}, and \ref{condRes}.  However, $T$ does not satisfy \ref{condKer} since $\Sigma^{ 2d }T$ is mapped to $0$ by $\overline{ \cF }( T,- ) = \H^0( - )$, but is not in $\add \Sigma^d T$.

Finally, the category $\overline{\cF}$ is stable under the functor $S_d = S_D \Sigma^{ -d }$ by \cite[thms.\ 2.16 and 2.21]{IO2}, where $S_D$ is again the Serre functor of $\Db( \modf \Lambda )$.  The functor $S_d$ plays the role of AR translation of $\overline{ \cF }$; in particular, it is an autoequivalence of $\overline{ \cF }$.  It follows that for each integer $\ell$, the object $S_d^{ \ell }T$ also satisfies \ref{condGhost}, \ref{condGhostDual}, and \ref{condRes}, but not \ref{condKer}.
\end{Example}

\medskip
\noindent
{\bf Acknowledgement.}
This work was supported by EPSRC grant EP/P016014/1 ``Higher Dimensional Homological Algebra''.  Karin M.\ Jacobsen is grateful for the hospitality of Newcastle University during the spring semester of 2017.

\end{document}